\documentclass[a4paper,11pt]{article}
\usepackage{latexsym,amssymb,amsmath,amsthm,amsxtra,xypic}
\usepackage{stmaryrd}
\usepackage{amsfonts}
\usepackage{amscd}
\usepackage{mathrsfs}
\usepackage[dvips]{graphicx}
\usepackage[all]{xy}

\newcommand{\allowpagebreak}
\allowdisplaybreaks

\newtheorem{thm}{Theorem}[section]

\newtheorem{pro}[thm]{Proposition}

\theoremstyle{definition}
\newtheorem{defi}[thm]{Definition}
\newtheorem{rmk}[thm]{Remark}

\newcommand {\emptycomment}[1]{}

\newcommand{\End}{\mathrm{End}}

\allowdisplaybreaks

\begin{document}
\title{Cohomology and automorphisms of Com-PreLie algebras}
\author{Tao Zhang, Ying-Hua Lu}

\date{}
\maketitle
 \maketitle

 \setcounter{section}{0}

  \begin{abstract}
  This paper introduces the concept of representations for Com-PreLie algebras and develops corresponding cohomology theories, examining how cohomology groups can be applied in the context of Com-PreLie algebras. Initially, we utilize the cohomology theory to investigate abelian extensions of Com-PreLie algebras. Next, given an abelian extension of Com-PreLie algebras and its representation, we explore the inducibility of Com-PreLie automorphisms, deriving both necessary and sufficient conditions for the inducibility problem. Lastly, we delve deeper into the inducibility of Com-PreLie automorphisms using the Wells exact sequences, offering a clear framework for studying the inducibility of Com-PreLie automorphisms.
\par\smallskip

\end{abstract}


\section{Introduction}

A Com-PreLie algebra is an important algebraic structure that appears in the mathematical formulation of rooted trees in the context of Quantum Field Theory and Fliess operators in Control Theory \cite{{Foissy1},{Foissy2}}.
Recall that Com-PreLie algebras are commutative algebras with an extra pre-Lie product, compatible with the product and pre-Lie product, see Definition \ref{Com-PreLie-alg} below.
Although some research has been conducted on Com-PreLie algebras, their representation theory and cohomology theory remain unexplored. This paper aims to address this gap. So the objective of this paper is to investigate the representations and cohomology of Com-PreLie algebras, as well as to highlight the applications of Com-PreLie cohomology in the realms of abelian extensions and inducibility problems (both of which will be touched upon briefly below).

\subsection{Abelian extensions}
The problem of extensions of algebraic structures, including central extensions, abelian extensions and non-abelian extensions, are useful to study the underlying structure. Among these types of extensions, the non-abelian extensions studied by Eilenberg and Maclane are the most general \cite{eilenberg}.
Specifically, a non-abelian extension of a Lie algebra \(\mathfrak{g}\) by another Lie algebra \(\mathfrak{h}\) can be described by a short exact sequence \(0 \rightarrow \mathfrak{h} \rightarrow \mathfrak{e} \rightarrow \mathfrak{g} \rightarrow 0\).
As a special case of non-abelian extensions, abelian extensions are even more specific. When a Lie algebra \(\mathfrak{g}\) is extended by a representation \(\mathfrak{h}\), if the Lie bracket on \(\mathfrak{h}\) is trivial and the induced representation on \(\mathfrak{h}\) coincides with the prescribed one, then this type of extension is called an abelian extension. In this case, the structure of \(\mathfrak{h}\) is relatively simple, which makes abelian extensions significant both in theory and application.

It is worth noting that the Chevalley-Eilenberg cohomology theory of Lie algebras provides a powerful tool for the classification of abelian extensions \cite{bardakov,hilton}. Through the study of cohomology groups, we can systematically understand and classify the abelian extensions of Lie algebras, thereby further revealing the intrinsic structure and properties of Lie algebras. Then numerous works have been devoted to abelian extensions of various kinds of
algebras, such as, Lie coalgebras, Lie Yagamuti algebras, pre-Lie
algebras, see \cite{du-tan, goswami,liu-chen}.
But little is known about the abelian
 extensions of Com-PreLie algebras. This is the first motivation for writing this paper.

\medskip
\subsection{The inducibility problems} The question of inducibility for automorphisms was initially explored in a work by Wells \cite{wells}. Following this, the same issue was extended to derivations in subsequent studies \cite{bai-zhang,tan-xu}.
In the realm of Lie algebras, the inducibility problem can be described as follows. Consider an abelian extension of a Lie algebra \(\mathfrak{g}\) by a given representation \(V\), represented by the exact sequence \(0 \rightarrow V \xrightarrow{i} E \xrightarrow{p} \mathfrak{g} \rightarrow 0\). Here, the induced representation on \(V\) from the abelian extension aligns with the specified representation. Let \(\mathrm{Aut}_V (E)\) denote the collection of all Lie algebra automorphisms \(\gamma \in \mathrm{Aut} (E)\) such that \(\gamma (V) \subset V\). There exists a group homomorphism \(\tau : \mathrm{Aut}_V (E) \rightarrow \mathrm{Aut} (V) \times \mathrm{Aut} (\mathfrak{g})\), defined by \(\tau (\gamma) = (\gamma |_V , p \gamma s)\), where \(s\) is a section of the map \(p\). The inducibility problem for automorphisms seeks to identify the necessary and sufficient conditions under which a pair \((\beta, \alpha) \in \mathrm{Aut} (V) \times \mathrm{Aut} (\mathfrak{g})\) of Lie algebra automorphisms falls within the image of \(\tau\).

It is evident that the aforementioned inducibility problems are applicable to other types of algebras as well. For further exploration of these issues across different algebras, one can refer to works such as \cite{ das-sahoo, goswami, mishra-das} and the references cited therein.

In this study, we aim to explore the inducibility problems within the framework of Com-PreLie algebras. To achieve this, we will develop appropriate Wells maps for Com-PreLie automorphisms and demonstrate that the challenges associated with these problems are rooted in the second Com-PreLie cohomology group. As a result, we derive a Wells exact sequence that links different Com-PreLie automorphism groups with the second Com-PreLie cohomology group.

\medskip
\subsection{Organization}
The paper is organized as follows.  In Section \ref{sec-2},  we recall the notion of Com-PreLie algebras. The representations and low dimensional
 cohomology of Com-PreLie algebras are then given. As applications of our cohomology, in Section \ref{sec-3} and \ref{sec-4}, we study abelian extensions and inducibility problems of a Com-PreLie algebra. In Section \ref{sec-5}, we construct the Wells type exact sequences to solve inducibility problem of automorphisms.

All vector spaces and algebras, (multi)linear maps, tensor products and wedge products are over a field of characteristic $0$ unless specified otherwise.


\medskip

\section{Representations and cohomology}\label{sec-2}
In this section, we recall some basic definitions and facts concerning Com-PreLie algebra. Further, we introduce  the concept of representations of a Com-PreLie algebra and define the semi-direct product in the context of Com-PreLie algebra. Finally, we define a low dimensional cohomology of a Com-PreLie algebra.

Recall that a \textbf{commutative associative algebra} is a vector space $A$ together with a bilinear map  $ \ast :A \times  A \rightarrow A$ such that
$$x\ast(y\ast z)=(x\ast y)\ast z,$$
$$x\ast y=y\ast x,$$
for all $x,y,z\in A,$ we denote it by $(A, \ast).$

A left \textbf{pre-Lie algebra} is a vector space $A$ together with a bilinear map  $ \bullet  :A \times  A \rightarrow A$ such that
$$(x \bullet y) \bullet z - x \bullet (y \bullet z) = (y \bullet x) \bullet z - y \bullet (x \bullet z),$$
for all $x,y,z\in A,$ we denote it by $(A, \bullet)$ or simply $A$.
For a pre-Lie algebra $A$ , the operation $[x, y] = x \bullet y - y \bullet x$ naturally induces a Lie algebra structure,
which is the origin of the term "pre-Lie algebra".

\begin{defi}
  A representation of a commutative associative algebra $(A,\ast)$ is a pair $(V, \mu)$, where $V$ is a vector space and
$\mu : A \rightarrow End(V )$ is a linear map satisfying
	$$\mu(x\ast y) = \mu(x)\ast \mu(y), \;\; \forall x,y\in A. $$
\end{defi}
\begin{defi}
  A representation of a pre-Lie algebra  $(A, \bullet)$ is a triple $(V,l,r)$, where $V$ is a vector space and
$l : A \times  V \rightarrow V, r : V \times  A \rightarrow V$  are linear maps satisfying
	$$l(x \bullet y)-l(x)l(y)=l(y \bullet x)-l(y)l(x),$$
$$r(y)l(x)-l(x)r(y)=r(y)r(x)-r(x\bullet y), \;\; \forall x,y\in A. $$
\end{defi}
In fact, the pair $(V,\mu)$ is a representation of a commutative associative algebra $(A, \ast)$ if and only if the direct sum $A \oplus V$ of
vector spaces is a (semi-direct product) commutative associative algebra by defining the multiplication $\ast_{\ltimes}$ on $A \oplus V $ by
\begin{eqnarray}\label{mul-ass}
(x + u) \ast_{\ltimes} (y + v) &=& x \ast y + \mu(x)v + \mu(y)u, \;\; \forall x,y \in A, u, v \in V .
\end{eqnarray}
We denote it by $A \ast_{\ltimes} V$.

The triple $(V,l,r)$ is a representation of a pre-Lie algebra $(A, \bullet)$ if and only if the direct sum $A \oplus V$ of
vector spaces is a (semi-direct product) pre-Lie algebra by defining the multiplication $\bullet_{\ltimes}$ on $A \oplus V $ by
\begin{eqnarray}\label{mul-pre}
(x + u) \bullet_{\ltimes} (y + v) &=& x \bullet y + l(x)v + r(y)u, \;\; \forall x,y \in A, u, v \in V .
\end{eqnarray}
We denote it by $A \bullet_{\ltimes} V$.

Let us recall the cohomology of pre-Lie algebras. Let $(V,l,r)$ be a representation of a pre-Lie algebra $(A, \bullet)$. Define the $n$-cochains group by
$$
C^{n}(A, V)=\operatorname{Hom}\left(\wedge^{n-1} A \otimes A, V\right), \quad(n \geq 1)
$$
and the coboundary operator $\partial: C^{n}(A, V) \longrightarrow C^{p+1}(A, V)$ is the Dzhumadil'daev coboundary operator \cite{D} given by
\begin{align*}
& ( \delta_\mathrm{D} f)\left(x_{1}, \cdots, x_{n+1}\right) \\
= & \sum_{i=1}^{n}(-1)^{i+1} l\left(x_{i}\right) f\left(x_{1}, \cdots, x_{i-1}, x_{i+1}, \cdots, x_{p+1}\right)  \\
& +\sum_{i=1}^{n}(-1)^{i+1} r\left(x_{p+1}\right) f\left(x_{1}, \cdots, x_{i-1}, x_{i+1}, \cdots, x_{n}, x_{i}\right) \\
& -\sum_{i=1}^{n}(-1)^{i+1} f\left(x_{1}, \cdots, x_{i-1}, x_{i+1}, \cdots, x_{n}, x_{i}\bullet x_{n+1}\right) \\
& +\sum_{1 \leq i<j \leq n}(-1)^{i+j} f\left(\left[x_{i}, x_{j}\right], x_{1}, \cdots, x_{i-1}, x_{i+1}, \cdots, x_{j-1}, x_{j+1}, \cdots, x_{p+1}\right)
\end{align*}
for any $x_{1}, \cdots, x_{n+1} \in A$ and $f \in C^{n}(A, V)$, where $\left[x_{i}, x_{j}\right]=x_{i} \bullet x_{j}-x_{j}\bullet x_{i}$.

\begin{defi}[\cite{Foissy1,Foissy2}]\label{Com-PreLie-alg}
A \textbf{ Com-PreLie algebra} is a family $A = (A, \ast, \bullet),$  where
$(A,\ast)$ is a commutative associative algebra and $(A,\bullet)$ is a pre-Lie algebra, $\ast$ and $\bullet$ are bilinear products on $A$, such that
\begin{eqnarray}\label{Com-PreLie}
x \bullet(y \ast z)&=& (x \bullet y) \ast z+y \ast (x \bullet z), \;\; \forall x,y,z \in A.
\end{eqnarray}
\end{defi}

\begin{rmk}
The origin Com-PreLie algebra defined in \cite{Foissy1,Foissy2} is a commutative associative algebra  $(A,\ast)$  and  a right pre-Lie algebra $(A,\bullet)$ satisfying the following condition:
\begin{eqnarray}\label{rCom-PreLie}
(x \ast y) \bullet z  &=& (x \bullet z) \ast y + x \ast (y \bullet z), \;\; \forall x,y,z \in A.
\end{eqnarray}
Thus the Com-PreLie algebra defined in this paper is in fact the left version of Com-PreLie algebra defined in \cite{Foissy1,Foissy2}.
\end{rmk}

\begin{defi}\label{biaoshi-Com-PreLie}
Let $A$ be a Com-PreLie algebra. A representation of $A$ is a quadruple $(V, \mu, l, r )$, such that
the pair $(V, \mu )$ is a representation of $(A,\ast)$, the triple $(V, l, r )$ is a representation of $(A, \bullet)$, and the following conditions
hold:
\begin{eqnarray}
\label{Com-repre-1}l(x)\mu(y)=\mu(x \bullet y)+\mu(y)l(x),\\
\label{Com-repre-2}r(x\ast y)=\mu(y)r(x)+\mu(x)r(y),
\end{eqnarray}
for all $x,y \in A.$
\end{defi}

Let $A$ be a Com-PreLie algebra. Then the quadruple $(A, \mu_\mathrm{ad}, l_\mathrm{ad},r_\mathrm{ad})$ is a representation, where $(\mu_\mathrm{ad})(x) y = x \ast y$, $(l_\mathrm{ad})(x) y = x \bullet y $ and $(r_\mathrm{ad})(x) y = y \bullet x $, for all $x, y \in A$. This is called the {\em adjoint representation} or the regular representation of $A$.
\begin{pro}
 Let $A$ be a Com-PreLie algebra. Let $V$ be a vector space and $\mu,l,r : A \rightarrow End(V )$ be linear maps. Then
$(V,\mu, l, r )$ is a representation of $A$ if and only if the direct sum $A \oplus V$ of vector spaces is a
Com-PreLie algebra  by defining the bilinear operations on $A \oplus V$ by Eqs. ~(\ref{mul-ass}) and (\ref{mul-pre})  respectively. We denote it by $A \ltimes V=(A \oplus V, \ast_{\ltimes}, \bullet_{\ltimes})$, and call it the semi-direct product.
\end{pro}
\begin{proof}
  Obviously, $(A \oplus V, \ast_{\ltimes} )$ is a commutative associative algebra and $(A \oplus V, \bullet_{\ltimes} )$ is a pre-Lie algebra.
  For all $x+u, y+v, z+w \in  A \oplus V,$  by (\ref{Com-PreLie}), (\ref{Com-repre-1}) and (\ref{Com-repre-2}), we
have
\begin{eqnarray*}
&&  (x+u) \bullet_{\ltimes}\big((y+v) \ast_{\ltimes} (z+w)\big) \\
&&- \big((x+u) \bullet_{\ltimes} (y+v)\big) \ast_{\ltimes} (z+w) \\
&&-  (y+v) \ast_{\ltimes} \big((x+u) \bullet_{\ltimes} (z+w)\big)\\
&=&x \bullet (y \ast z) + l(x)\mu(y)w+l(x)\mu(z)v + r(y \ast z)u\\
&&- (x \bullet y) \ast z - \mu(x \bullet y)w - \mu(z)l(x)v-\mu(z)r(y)u\\
&&-y \ast (x \bullet z) - \mu(y)l(x)w-\mu(y)r(z)u - \mu(x \bullet z)v\\
&=&0.
\end{eqnarray*}
This finishes the proof.
 \end{proof}

Now we define a cohomology for Com-PreLie algebras similar as the Flato-Gerstenhaber-Voronov cohomology for Poisson algebras in \cite{FGV}.
Let $A$ be a Com-PreLie algebra and $V$ be a representation of it.
A linear map $f \in \mathrm{Hom} ({A}^{\otimes m} \otimes (\wedge^{n-1} {A} \otimes A), V)$ is called a {\em $(m,n)$-cochain} if
\begin{align*}
    \sum_{\sigma \in \mathrm{Sh}_{(i, m-i)}} (-1)^\sigma ~ f \big(  ( y_{\sigma^{-1} (1)} \otimes \cdots \otimes y_{\sigma^{-1} (i)} \otimes y_{\sigma^{-1} (i+1)} \otimes \cdots \otimes y_{\sigma^{-1} (m)} ) \otimes (x_1 \otimes \cdots \otimes x_n )   \big) = 0,
\end{align*}
for all $0 < i < m$ and $y_1, \ldots, y_m, x_1, \ldots, x_n \in {A}$. We denote the set of all $(m, n)$-cochains by $C^{m,n} ({A}, V)$. Then for any $k \geq 0$, the space of $k$-cochains is defined by $ C^k ({A}, V) := \bigoplus_{\substack{m+n = k \\ m \neq 1}} C^{m, n} ({A}, V)$
and the coboundary map $\delta^{k} :  C^k ({A}, V)  \rightarrow  C^{k+1} ({A}, V)$ is given by
\begin{align*}
    \delta^{k} (f) = \delta_\mathrm{H}^{m,n} (f) + (-1)^m ~ \! \delta_\mathrm{D}^{m,n} (f),
\end{align*}
for $ f \in C^{m,n}({A}, V)$ with $m + n = k ~ (m \neq 1)$. We define the map $\delta_\mathrm{H}^{m,n} : C^{m,n} ({A}, V) \rightarrow C^{m+1, n} ({A}, V)$ to be the Harrison coboundary operator
\begin{align*}
     &( \delta_\mathrm{H}^{m,n} f ) ( (y_1 \otimes \cdots \otimes y_{m+1}) ~\otimes~ (x_1 \otimes \cdots \otimes x_n) )\\
     =& \mu_{y_1} f ( (y_2 \otimes \cdots \otimes y_{m+1})\otimes (x_1 \otimes \cdots \otimes x_n) ) \\
     &+ \sum_{i=1}^m (-1)^i f (  (y_1 \otimes \otimes y_i \ast y_{i+1} \otimes \cdots \otimes y_{m+1}) \otimes (x_1 \otimes \cdots \otimes x_n) ) \\
     &+ (-1)^{m+1} ~ \!  \mu_{y_{m+1}}f ( (y_1 \otimes \cdots \otimes y_{m})\otimes (x_1 \otimes \cdots \otimes x_n) )
\end{align*}
and the coboundary operator $\delta^{m,n}_\mathrm{D} : C^{m,n} ({A}, V) \rightarrow C^{m, n+1} ({A}, V)$ by
\begin{align*}
    &(\delta^{m,n}_\mathrm{D} f)  (  (y_1 \otimes \cdots \otimes y_{m})\otimes (x_1 \otimes \cdots \otimes x_{n+1})) \\
    &= \sum_{i=1}^{n} (-1)^{i+1} l(x_i) f ( (y_1 \otimes \cdots \otimes y_{m}) \otimes (x_1 \otimes \cdots \otimes \widehat{x_i} \otimes \cdots  \otimes x_{n+1}) )  \\
    &\quad +\sum_{i=1}^{n}(-1)^{i+1} r\left(x_{n+1}\right) f\left((y_1 \otimes \cdots \otimes y_{m}) \otimes
    (x_{1}\otimes \cdots\otimes x_{i-1}\otimes x_{i+1}\otimes \cdots\otimes x_{n}\otimes x_{i})\right) \\
&\quad -\sum_{i=1}^{n}(-1)^{i+1} f\left((y_1 \otimes \cdots \otimes y_{m}) \otimes
(x_{1}\otimes \cdots\otimes x_{i-1}\otimes x_{i+1}\otimes\cdots\otimes x_{n}\otimes x_{i}\bullet x_{n+1})\right) \\
    & \quad - \sum_{j=1}^m f (  (y_1 \otimes \cdots \otimes  x_i\bullet y_j \otimes \cdots \otimes y_{m}) \otimes (x_1 \otimes \cdots \otimes \widehat{x_i} \otimes \cdots  \otimes x_{n+1})    ) \\
    & \quad + \sum_{1 \leq i < j \leq n} (-1)^{i+j} f (   (y_1 \otimes \cdots \otimes y_{m}) \otimes ([ x_i, x_j] \otimes x_1 \otimes \cdots \otimes \widehat{x_i} \otimes \cdots \otimes \widehat{x_j} \otimes \cdots \otimes x_{n+1})  ),
\end{align*}
for $f \in C^{m,n} (A, V)$, $y_1, \ldots, y_{m}, x_1, \ldots, x_{n+1} \in A$ and $[ x_i, x_j] =x_i\bullet x_j - x_j\bullet x_i $. Then it turns out that $\{ C^\bullet (A, V), \delta\}$ is a cochain complex. The set of all $n$-cocycles is denoted by $Z^n (A, V)$ and the set of all $n$-coboundaries is denoted by $B^n (A, V)$. The corresponding cohomology group is called the  cohomology of the Com-PreLie algebra $A$ with coefficients in the representation $V$, and is denoted by $H^\bullet (A, V)$.

For example, we denote the space of $k$-cochains $ C^k (A,V)$ by $C^{k}$, i.e.
\begin{eqnarray*}
C^{1} &=& \mathrm{Hom}(A,V), \\
C^{2} &=& \mathrm{Hom}(A\times A,V),\\
   C^{3}  &=& \mathrm{Hom}(A\times A\times A,V) .
   \end{eqnarray*}
For a 1-cochain $ N\in  C^{1}$, we define the
  map $\delta^1:C^{1} \to C^{2}  $ by
 \begin{eqnarray*}
\label{1co-V-1}\delta^1_{H} N(x,y) &=&\mu(x) N(y)-N(x \ast y)+\mu(y)N(x),\\
\label{1co-V-2}\delta^1_{D} N(x,y) &=&l(x)N(y)-N(x \bullet y)+r(y)N(x),
\end{eqnarray*}
 Thus a 1-cocycle is $ N\in  \mathrm{Hom}(A,V)$ such that
\begin{eqnarray}
&&\mu(x) N(y)-N(x \ast y)+\mu(y)N(x)=0,\label{A1co-1}\\
&&l(x)N(y)-N(x \bullet y)+r(y)N(x)=0.\label{A1co-2}
\end{eqnarray}
For a 2-cochain $ (\phi, \psi)\in  C^{2}$, where $\phi$ is a symmetric map,  we define the
  map $\delta^2:C^{2} \to C^{3}  $ by
   \begin{eqnarray*}
\label{A2co-V-1}  \delta^2_{H}\phi(x,y,z)
&=&\phi(x,y\ast z)+\mu(x)\phi(y,z)-\phi(x\ast y,z) -\mu(z)\phi(x,y),  \\
\label{A2co-V-2}  \delta^2_{D}\psi(x,y,z)
&=&\psi( x\bullet y, z)+r(z)\psi(x,y)-\psi(x, y \bullet z )-l(x)\psi(y,z)\notag\\
&&- \psi(y\bullet x, z)-r(z)\psi(y,x)+\psi(y, x\bullet z)+l(y)\psi(x,z),\\
\label{A2co-V-3} ( \delta^2_{H}\phi +\delta^2_{D} \psi)(x,y,z)
&=&\psi(x, y \ast z)-\mu(z)\psi(x,y)-\mu(y)\psi(x,z)\notag\\
&&-\phi(y,x \bullet z)+l(x)\phi(y,z)-\phi(x \bullet y,z).
\end{eqnarray*}

 Thus a 2-coboundary is $(\phi, \psi )\in C^{2}$
such that $(\phi, \psi)=\delta_1(N),$ i.e.
   \begin{eqnarray}
  \label{A1} &&\phi(x,y)=\mu(x) N(y)-N(x \ast y)+\mu(y)N(x),\\
   \label{A2}&&\psi(x,y)=l(x)N(y)-N(x \bullet y)+r(y)N(x),
   \end{eqnarray}
  and a 2-cocycle is a pair$ (\phi, \psi)\in  \mathrm{Hom}(A \times A,V)$ such that
\begin{eqnarray}
\label{A2co-1}&&\phi(x,y\ast z)+\mu(x)\phi(y,z)-\phi(x\ast y,z) -\mu(z)\phi(x,y)=0,\\
\label{A2co-2}&&\psi( x\bullet y, z)+r(z)\psi(x,y)-\psi(x, y \bullet z )-l(x)\psi(y,z)\notag\\
&&- \psi(y\bullet x, z)-r(z)\psi(y,x)+\psi(y, x\bullet z)+l(y)\psi(x,z)=0,\\
\label{A2co-3}&&\psi(x, y \ast z)+l(x)\phi(y,z)-\phi(x \bullet y,z)-\mu(z)\psi(x,y)\notag\\
&&-\phi(y,x \bullet z)-\mu(y)\psi(x,z)=0.
\end{eqnarray}
Further, two  $2$-cocycles $(\phi, \psi)$ and $(\phi', \psi')$ are {\em cohomologous} if there exists a linear map $f : A \rightarrow V$ such that for all $x, y \in A,$
\begin{eqnarray}
    \label{cobound-1} \phi(x, y) - \phi' (x, y) = \mu(x) f (y)- f (x \ast y ) + \mu(y) f (x), \\
     \label{cobound-2}\psi (x, y) - \psi' (x, y) =l(x) f (y) - f (x\bullet y ) + r(y) f(x).
\end{eqnarray}
\begin{pro}
The space of 2-coboundaries is contained in the space of 2-cocycles, i.e. $D^2 = D_2 \cdot  D_1 = 0.$
\end{pro}
\begin{proof}
  We should verify that if $(\phi, \psi)= D_1(N)$, then
$D_2(\phi,\psi )= D_2D_1(N) = 0.$ Assume
$(\phi, \psi )$ is given by (\ref{A1})-(\ref{A2}), we will verify that it must satisfy Eqs.~(\ref{A2co-1})-(\ref{A2co-3}). We compute Eq.~(\ref{A2co-3}) as follows:
\begin{eqnarray*}
&&\psi(x, y \ast z)+l(x)\phi(y,z)-\phi(x\bullet y,z)-\mu(z)\psi(x,y)-\phi(y,x\bullet z)-\mu(y)\psi(x,z)\\
&=& \underbrace{l(x)N(y \ast z)}_E  \underbrace{-N(x \bullet (y \ast z))}_A\underbrace{+r(y \ast z)N(x)}_D\\
&&\underbrace{l(x)\mu(y) N(z)}_C\underbrace{-l(x)N(y \ast z)}_E\underbrace{+l(x)\mu(z)N(y) }_B\\
&&\underbrace{-\mu(x\bullet y) N(z)}_C\underbrace{+N((x\bullet y) \ast z)}_A\underbrace{-\mu(z)N(x\bullet y)}_F\\
&&\underbrace{-\mu(z)l(x)N(y)}_B\underbrace{+\mu(z)N(x \bullet y)}_F\underbrace{-\mu(z)r(y)N(x)}_D\\
&&\underbrace{-\mu(y) N(x\bullet z)}_G\underbrace{+N(y \ast (x\bullet z))}_A\underbrace{-\mu(x\bullet z)N(y)}_B\\
&&\underbrace{-\mu(y)l(x)N(z)}_C\underbrace{+\mu(y)N(x \bullet z)}_G\underbrace{-\mu(y)r(z)N(x)}_D\\
&=&0
\end{eqnarray*}
where we use the fact that $A$ a Com-PreLie algebra and $V$ is a representation of the Com-PreLie algebra $A$.
The other equalities can be checked similarly. This completes the proof.
\end{proof}

\begin{defi}
The second cohomology group of $A$ with coefficients in $V$ is defined as the quotient
$$\mathcal{H}^2(A ,V)=\mathcal{Z }^2 (A ,V)/\mathcal{B }^2 ( A ,V),$$
where the set of 2-cocycles by $\mathcal{Z }^2(A ,V )$, the set of 2-coboundaries by  $ \mathcal{B}^2 (A ,V)$ and the second cohomology group by $\mathcal{H}^2 (A,V)$.
\end{defi}

\section{Abelian extensions} \label{sec-3}
In this section, we study abelian extensions of a Com-PreLie algebra by a given representation.
We classify abelian extensions of a Com-PreLie algebra $(A, \ast, \bullet)$ via the second cohomology group with coefficients in a representation.

Let $A$ be a Com-PreLie algebra and let $V$ be a vector spaces such that $V$ is the trivial Com-PreLie algebra.

\begin{defi}
    An \textbf{ abelian extension} of a Com-PreLie algebra $A$ by a vector spaces $V$ is a short exact sequence
\begin{eqnarray}\label{ab-ext-1}
\mathcal{E}:\xymatrix{
0 \ar[r] & V  \ar[r]^{i}  &  \widehat{A}  \ar[r]^{j} & A  \ar[r] & 0
}
\end{eqnarray}
of Com-PreLie algebras.
\end{defi}
Thus, an abelian extension is given by a Com-PreLie algebra $A$ equipped with homomorphisms $i: V \rightarrow \widehat{A} $ and $j:  \widehat{A} \rightarrow A$ of Com-PreLie algebras such that (\ref{ab-ext-1}) is a short exact sequence.

Consider an abelian extension as of (\ref{ab-ext-1}). A \textbf{section} of the map $j:  \widehat{A}  \rightarrow A $ is given by a linear map $s : A \rightarrow \widehat{A}$ satisfying $j \circ s = \mathrm{id}_A$. Let $s$ be any section of the map $j$. Then we define maps
$$(\mu,l,r):  A \to \End (V),$$
 respectively by
 $$\mu(x)u= s(x) \ast_{\widehat{A}} u,\quad l(x)u=s(x) \bullet_{\widehat{A}} u,\quad r(x)u=u \bullet_{\widehat{A}} s(x),$$
 for all $x \in A$, $u \in V$. Then we have the following result.
\begin{pro}\label{rep-pro}
 With the above notations, the  quadruple $(V, \mu, l ,r)$ is a representation of $A.$
    Moreover, this representation is independent on the choice of sections.
\end{pro}

\begin{proof}
Firstly, we show that the definition is well-defined. Since $\mathrm{ker} (j) \cong V$, then for $u \in V$, we have $j(u) = 0.$
By the fact that $j$ is a homomorphism between $\widehat{A}$ and $A$, we get
$$j(\mu(x)u)=j(s(x) \ast_{\widehat{A}} u)=j(s(x)) \ast j(u)=j(s(x)) \ast 0=0.$$
Thus $\mu(x)u \in \mathrm{ker} (j) \cong V.$ Similar computations show that
$$j(l(x)u)=j(s(x) \bullet_{\widehat{A}} u)=j(s(x)) \bullet j(u)=j(s(x)) \bullet 0=0,$$
$$j(r(x)u)=j(u \bullet_{\widehat{A}} s(x))=j(u) \bullet j(s(x))=0 \bullet j(s(x))=0.$$
Thus $l(x)u, r(x)u \in \mathrm{ker} (j) \cong V.$

We will show that these maps are independent of the choice of $s$. In fact, if we choose another section
$s': A \rightarrow \widehat{A}$, then $j(s(x) - s'(x)) = x-x = 0$, i.e. $s(x) - s'(x)\in \mathrm{ker} (j) \cong V$. Thus we get $s(x) \ast_{\widehat{A}} u=s'(x) \ast_{\widehat{A}} u$, $s(x) \bullet_{\widehat{A}} u=s'(x) \bullet_{\widehat{A}} u$ and $u \bullet_{\widehat{A}} s(x)=u \bullet_{\widehat{A}} s'(x)$,   which implies that the maps are independent on the choice of $s$. Therefore the representation structures are
well-defined.

Secondly, we check that $(V, \mu, l ,r)$ is a representation of $A.$
To establish this theory, verifying all conditions in Definition \ref{biaoshi-Com-PreLie} is essential, and this verification is contingent upon the fact that $\widehat{A}=(\widehat{A} , \ast_{\widehat{A}},\bullet_{\widehat{A}} )$ is a Com-PreLie algebra.

By the fact that $(\widehat{A},\ast_{\widehat{A}})$ is an associative algebra, we have
 \begin{eqnarray*}
&& \mu(x  \ast y)u\\
&=&s(x \ast  y)\ast_{\widehat{A}}(u)\\
&=&(s(x)\ast_{\widehat{A}}s(y))\ast_{\widehat{A}}(u)\\
&=&s(x)\ast_{\widehat{A}}(s(y)\ast_{\widehat{A}}(u))\\
&=& \mu(x)\mu(y) u,
 \end{eqnarray*}
which implies that $(V, \mu)$ is a representation of $(A,\ast)$.

Since $(\widehat{A},\bullet_{\widehat{A}})$ is a pre-Lie algebra, we get
 \begin{eqnarray*}
&&l(x \bullet y)u-l(x)l(y)u\\
&=&s(x \bullet y) \bullet_{\widehat{A}} u-s(x) \bullet_{\widehat{A}}(s(y) \bullet_{\widehat{A}} u)\\
&=&(s(x)\bullet_{\widehat{A}}s(y)) \bullet_{\widehat{A}} u-s(x) \bullet_{\widehat{A}}(s(y) \bullet_{\widehat{A}} u)\\
&=&(s(y)\bullet_{\widehat{A}}s(x)) \bullet_{\widehat{A}} u-s(y) \bullet_{\widehat{A}}(s(x) \bullet_{\widehat{A}} u)\\
&=& s(y \bullet x) \bullet_{\widehat{A}} u-s(y) \bullet_{\widehat{A}}( s(x) \bullet_{\widehat{A}} u)\\
&=&l(y \bullet x)u-l(y)l(x)u.
\end{eqnarray*}
Similarly, we also obtain
$$r(y)l(x)-l(x)r(y)=r(y)r(x)-r(x\bullet y).$$
This verify $(V, l, r )$ is a representation of $(A, \bullet).$

For the compatibility conditions (\ref{Com-repre-1}) and (\ref{Com-repre-2}), by direct computation
 \begin{eqnarray*}
 &&\mu(x \bullet y)u+\mu(y)l(x)u\\
 &=&s(x \bullet y)\ast_{\widehat{A}}u +s(y)\ast_{\widehat{A}}(s(x) \bullet_{\widehat{A}} u)\\
 &=& (s(x)\bullet_{\widehat{A}} s(y))\ast_{\widehat{A}}u+s(y)\ast_{\widehat{A}}(s(x) \bullet_{\widehat{A}} u)\\
 &=&s(x)\bullet_{\widehat{A}} (s(y)\ast_{\widehat{A}} u)\\
 &=&l(x)\mu(y)u,
 \end{eqnarray*}
 where we exploit the fact that $\widehat{A}$ is a Com-PreLie algebra. In the same manner, (\ref{Com-repre-2}) is fulfilled as well.
 The proof is finished.
\end{proof}

\begin{defi}
    Two abelian extensions (two rows of the diagram (\ref{abel-diag})) are said to be \textbf{isomorphic} if there exists an isomorphism $F: \widehat{A} \rightarrow \widehat{A}' $ of Com-PreLie algebra making the below diagram commutative
    \begin{eqnarray}\label{abel-diag}
        \xymatrix{
0 \ar[r] & V \ar@{=}[d] \ar[r]^{i}  &  \widehat{A} \ar[d]^{F} \ar[r]^{j} & A  \ar[r] \ar@{=}[d] & 0 \\
0 \ar[r] & V  \ar[r]_{i'}  &  \widehat{A}' \ar[r]_{j'} & A \ar[r] & 0.
}
    \end{eqnarray}
\end{defi}

Let $A$ be a Com-PreLie algebra and $(V, \mu,l,r)$ be a given representation of it. We define $\mathrm{Ext} (A ,V)$ to be the set of all isomorphism classes of abelian extensions of $A$ by the vector spaces $V$ for which the induced representation is the prescribed one. Then we have the following result.

\begin{pro}
 Let $A$ be a Com-PreLie algebra and $(V, \mu,l,r)$ be a representation of it. Then there is a well-defined map
 \begin{align*}
   \delta : \mathrm{Ext}(A ,V) \rightarrow \mathcal{H}^2 (A ,V).
\end{align*}
\end{pro}
\begin{proof}
Let (\ref{ab-ext-1}) be an abelian extension. For any section $s$ of the map $j$, we define two bilinear maps
    \begin{eqnarray*}
  ~ \phi :A \times A \rightarrow V ,\\
    ~ \psi: A \times A \rightarrow V ,
    \end{eqnarray*}
 by
 $$\phi(x,y)=s(x) \ast_{\widehat{A}} s(y) -s(x\ast y),$$
 $$\psi(x,y)=s(x) \bullet_{\widehat{A}} s(y) -s(x\bullet y),$$
 for all $x, y \in A$. Then  $(\phi,\psi) \in C^2$ defined by the above notations is a $2$-cocycle of $A$ with coefficients in $V$, where the representation is given by Proposition \ref{rep-pro}, i.e., $(\phi,\psi)$ satisfies the Eqs.~(\ref{A2co-1})-(\ref{A2co-3}). By $\widehat{A}$ is a Com-PreLie algebra that satisfies the compatibility conditions.
We have the equality
$$ s(x) \bullet_{\widehat{A}} (s(y) \ast_{\widehat{A}} s(z)) = (s(x) \bullet_{\widehat{A}} s(y)) \ast_{\widehat{A}} s(z) + s(y) \ast_{\widehat{A}} (s(x) \bullet_{\widehat{A}} s(z)),$$
   we get that the right hand side is equal to
     \begin{eqnarray*}
     && (s(x) \bullet_{\widehat{A}} s(y)) \ast_{\widehat{A}} s(z) + s(y) \ast_{\widehat{A}} (s(x) \bullet_{\widehat{A}} s(z))\\
     &=&(\psi(x,y)+s(x\bullet y))\ast_{\widehat{A}} s(z)+ s(y) \ast_{\widehat{A}} (\psi(x,z)+s(x\bullet z))\\
     &=&\psi(x,y)\ast_{\widehat{A}} s(z)+s(x\bullet y)\ast_{\widehat{A}} s(z)+s(y) \ast_{\widehat{A}} \psi(x,z)+s(y) \ast_{\widehat{A}} s(x\bullet z)\\
     &=&\mu(z)\psi(x,y)+\phi(x\bullet y,z)+s((x\bullet y)\ast z)+\mu(y)\psi(x,z)+\phi(y,x\bullet z)+s(y\ast (x\bullet z))
     \end{eqnarray*}
Similarily, the left hand side is equal to
$$ l(x) \phi(y,z)+\psi(x,y\ast z)+s(x\bullet (y\ast z))$$
  Thus we have
 \begin{eqnarray*}
  &&l(x) \phi(y,z)+\psi(x,y\ast z)\\
  &=&\mu(z)\psi(x,y)+\phi(x\bullet y,z)+\mu(y)\psi(x,z)+\phi(y,x\bullet z),
  \end{eqnarray*}
   which implies that Eq.~(\ref{A2co-3}) holds. By similar calculations, it can be verified that Eqs.~(\ref{A2co-1})-(\ref{A2co-2}) also hold.
  Hence the abelian extension (\ref{ab-ext-1}) gives rise to a cohomology class in $\mathcal{H}^2( A ,V)$. Moreover, the cohomology class does not depend on the choice of the section.

 Next, we consider two equivalent abelian extensions as of (\ref{abel-diag}).
 If $s$ is any section of the map $j$, then we have $j' \circ (F \circ s) = j \circ s = \mathrm{id}_A$.
 This shows that $(F \circ s)$ is a section of the map $j'$.
 If $(\phi'.\psi')$ is the $2$-cocycle corresponding to the second abelian extension and its section $F \circ s$ then
\begin{eqnarray*}
   \phi' (x,y)
  &=& (F \circ s)(x)\ast_{\widehat{A}}(F \circ s)(y)-(F \circ s)(x \ast y)\\
  &=& F \big( s(x)\ast_{\widehat{A}}s(y)-s(x \ast y)  \big) \\
  &=& s(x)\ast_{\widehat{A}}s(y)-s(x \ast y) \quad (\because ~ F|_V = \mathrm{id}_V) \\
  &=& \phi(x,y).
   \end{eqnarray*}
    Similarly, $\psi'(x,y)=\psi(x,y).$
Hence we have $(\phi,\psi) = (\phi',\psi')$ and thus they corresponds to the same element in $\mathcal{H}^2( A ,V)$. Thus, we obtain a well-defined map
\begin{align*}
      \delta : \mathrm{Ext}(A ,V) \rightarrow \mathcal{H}^2(A ,V).
\end{align*}

\end{proof}

\begin{pro}
 Let $A$ be a Com-PreLie algebra and $V$ be a representation of it. Then there is a well-defined map
 \begin{align*}
   \Omega :  \mathcal{H}^2 (A ,V) \rightarrow \mathrm{Ext}(A ,V).
\end{align*}
\end{pro}
\begin{proof}
Let $(\phi,\psi)$ be a $2$-cocycle.
 Take $\widehat{A} = A \oplus V$. Then we define maps
     \begin{eqnarray*}
     ~\ast_{\widehat{A}}:\widehat{A} \times \widehat{A} \longrightarrow  \widehat{A},\\
     ~\bullet_{\widehat{A}}:\widehat{A}\times \widehat{A} \longrightarrow \widehat{A},\\
    \end{eqnarray*}
 by
$$ (x+u)\ast_{\widehat{A}}(y+v)=x\ast y+\mu(x)v+\mu(y)u+\phi(x,y),$$
$$ (x+u)\bullet_{\widehat{A}}(y+v)=x\bullet y+l(x)v+r(y)u+\psi(x,y),$$
 for all $ x+u,y+v \in \widehat{A}$.  Then $\widehat{A}$ defined by the above notations is a Com-PreLie algebra.
 Now it suffices to verify the fundamental identity. For all $ x+u,y+v, z+w \in \widehat{A}$, we have
  \begin{eqnarray*}
 && (x+u)\bullet_{\widehat{A}} \big((y+v) \ast_{\widehat{A}} (z+w)\big)\\
 &=&(x+u)\bullet_{\widehat{A}} (y \ast z+\mu(y)w+\mu(z)v+\phi(y,z) )\\
 &=&x\bullet (y \ast z)+l(x)\mu(y)w+l(x)\mu(z)v+l(x)\phi(y,z)+r(y \ast z)u+\psi(x,y \ast z),
 \end{eqnarray*}
and
 \begin{eqnarray*}
 &&\big((x+u) \bullet_{\widehat{A}} (y+v)\big) \ast_{\widehat{A}} (z+w) + (y+v) \ast_{\widehat{A}} \big((x+u) \bullet_{\widehat{A}} (z+w)\big)\\
   &=&(x\bullet y+l(x)v+r(y)u+\psi(x,y))\ast_{\widehat{A}} (z+w) +(y+v) \ast_{\widehat{A}} (x\bullet z+l(x)w+r(z)u+\psi(x,z))\\
   &=&(x\bullet y)\ast z+\mu(x\bullet y)w+\mu(z)l(x)v+\mu(z)r(y)u+\mu(z)\psi(x,y)+\phi(x\bullet y,z)\\
   &&+ y\ast (x\bullet z)+\mu(y)l(x)w+\mu(y)r(z)u+\mu(y)\psi(x,z)+\mu(x\bullet z)v+\phi(y,x\bullet z)
 \end{eqnarray*}
It follows that
$$ (x+u)\bullet_{\widehat{A}} \big((y+v) \ast_{\widehat{A}} (z+w)\big)=
\big((x+u) \bullet_{\widehat{A}} (y+v)\big) \ast_{\widehat{A}} (z+w) + (y+v) \ast_{\widehat{A}} \big((x+u) \bullet_{\widehat{A}} (z+w)\big)$$
 since $(\phi,\psi)$ be a $2$-cocycle and $V$ is a represention of $A.$
Then by direct computation, we have that $(\widehat{A},\ast_{\widehat{A}} ) $
is a commutative associative algebra and $(\widehat{A},\bullet_{\widehat{A}} )$ is a pre-Lie algebra.

  Moreover,
 \begin{align*}
\mathcal{E}:\xymatrix{
0 \ar[r] & V  \ar[r]^{i}  &  \widehat{A} \ar[r]^{j} & A  \ar[r] & 0
}
\end{align*}
is an abelian extension, where the above maps are given by $i (v) = (0, v)$ and $j (x, v) = x.$
Next, let $(\phi',\psi')$ be another $2$-cocycle cohomologous to $(\phi,\psi)$, then there exists a linear map $f : A \rightarrow V$ satisfying Eqs.~ (\ref{cobound-1})-(\ref{cobound-2}). We define a linear map $F : \widehat{A} \rightarrow \widehat{A}'$  by
\begin{align*}
    F (x+ u) = x+ u + f(x) ,
\end{align*}
for all $x+ u \in \widehat{A}$. We have
\begin{eqnarray*}
&&F\big((x+u)\ast_{\widehat{A}}(y+v)\big)\\
&=&F\big( x\ast y +\mu(x)v+\mu(y)u+\phi(x,y) \big) \\
&=&x\ast y +\mu(x)v+\mu(y)u+\phi(x,y)+f( x\ast y),
\end{eqnarray*}
and
\begin{eqnarray*}
&&F(x+u)\ast_{\widehat{A'}} F(y+v)\\
&=&(x+u + f (x))\ast_{\widehat{A'}}(y+v+f(y))\\
&=&x\ast y+\mu(x)v +\mu(x) f(y)+ \mu(y)u+\mu(y)f(x)+\phi'(x,y).
\end{eqnarray*}
It follows that
$$F\big((x+u)\ast_{\widehat{A}}(y+v)\big)=F(x+u)\ast_{\widehat{A'}} F(y+v).$$

Similarly,
$$F\big((x+u)\bullet_{\widehat{A}}(y+v)\big)=F(x+u)\bullet_{\widehat{A'}} F(y+v).$$

Then the map $F$ defined by the above notations is an isomorphism of the abelian extensions from $\widehat{A}$ to $\widehat{A}'$.

Hence there is a well-defined map  $\Omega :  \mathcal{H}^2 (A ,V) \rightarrow \mathrm{Ext}(A ,V)$. This completes the proof.
\end{proof}
 As a consequence, we obtain the following classification result of abelian extensions.
 \begin{thm}
    Let $A$ be  a Com-PreLie algebra and $V$ be a representation of it. Then
    \begin{align*}
        \mathrm{Ext} (A,V) \cong \mathcal{H}^2 (A,V).
    \end{align*}
\end{thm}

\section{Inducibility of Com-PreLie automorphisms}\label{sec-4}
In this section, we study the inducibility problem for a pair of Com-PreLie algebra automorphisms in a given abelian extension of a Com-PreLie algebra by a representation. To find an answer to this problem, we define the Wells map in the context of Com-PreLie algebras and show that a pair of Com-PreLie algebra automorphisms is inducible if and only if its image under the Wells map (the image lies in the second cohomology group) vanishes identically.

Let
\begin{eqnarray*}
\mathcal{E}:\xymatrix{
0 \ar[r] & V  \ar[r]^{i}  &  \widehat{A}  \ar[r]^{j} & A  \ar[r] & 0
}
\end{eqnarray*}
    be an abelian extension of the Com-PreLie algebra $A$ by a representation $(V, \mu, l,r)$.

 Let $\mathrm{Aut}_{V} (\widehat{A})$ be the group of all Com-PreLie algebra automorphisms $\gamma \in \mathrm{Aut} (\widehat{A})$ that satisfies $\gamma ({V}) \subset V$. Hence, for any $\gamma \in \mathrm{Aut}_{V} (\widehat{A})$ we naturally have $\gamma \big|_{V} \in \mathrm{Aut} (V)$ a Com-PreLie algebra automorphism. Next, for any linear section $s: A \rightarrow \widehat{A}$ of the map $j$ (i.e., $j \circ s = \mathrm{id}_{A}$), we can also define a map $\overline{\gamma} : A \rightarrow A$ by $\overline{\gamma} (x) := j \gamma s (x)$, for $x \in A$. It can be easily checked that the map $\overline{\gamma}$ doesn't depend on the choice of the section $s$. Moreover, $\overline{\gamma}$ is a bijection on the set $j$. Let $( \phi, \psi)$ be the Com-PreLie $2$-cocycle corresponding to the given abelian extension and induced by the section $s$. Then
\begin{align*}
    \phi (x, y) := s(x) \ast_{\widehat{A}} s(y) - s (x \ast y) \quad  \text{ and } \quad \psi (x, y): = s(x) \bullet_{\widehat{A}} s(y)- s(x\bullet y ), \text{ for } x, y \in A.
\end{align*}
Hence we get
\begin{align*}
    \overline{\gamma} (x \ast y) = j \gamma ( s (x \ast y)) =~& j \gamma \big(  s(x) \ast_{\widehat{A}} s (y) - \phi (x, y) \big) \\
    =~& j \gamma \big(  s(x) \ast_{\widehat{A}} s (y)  \big) \quad (\because ~ \gamma (V) \subset V ~~ \text{ and }~~ j \big|_{V} = 0) \\
   =~& j \gamma s (x) \ast j \gamma s (y) = \overline{\gamma} (x) \ast \overline{\gamma} (y).
\end{align*}
In the same way, we can show that $\overline{\gamma}  (x \bullet y) =  \overline{\gamma} (x)\bullet \overline{\gamma} (y) $, for all $x, y \in A$. Thus, $\overline{\gamma} : A \rightarrow A$ is a Com-PreLie algebra automorphism, i.e., $\overline{\gamma} \in \mathrm{Aut} (A)$. Hence, we obtain a group homomorphism
\begin{align*}
    \tau : \mathrm{Aut}_{V} (\widehat{A}) \rightarrow \mathrm{Aut} (V) \times \mathrm{Aut} (A) ~~ \text{ given by }~~ \tau (\gamma) := (\gamma \big|_{V}, \overline{\gamma}).
\end{align*}
Keeping this in mind, we say that a pair of Com-PreLie algebra automorphisms $(\beta, \alpha) \in \mathrm{Aut} (V) \times \mathrm{Aut} (A)$ is {\bf inducible} if this pair lies in the image of the map $\tau$. Equivalently, $(\beta, \alpha)$ is inducible if there exists a Com-PreLie algebra automorphism $\gamma \in \mathrm{Aut}_{V} (\widehat{A})$ such that $\gamma \big|_{V} = \beta$ and $\overline{\gamma} = \alpha$.

Our aim in this section is to find a necessary and sufficient condition (independent of the choice of any section) for a pair of Com-PreLie algebra automorphisms in $\mathrm{Aut} (V) \times \mathrm{Aut} (A)$ to be inducible. We will start with the following result.

\begin{pro}\label{Idu}
Let (\ref{ab-ext-1}) be an abelian extension of the Com-PreLie algebra $A$ by a representation $(V, \mu,l, r)$. The pair $(\beta, \alpha) \in \mathrm{Aut} (V) \times \mathrm{Aut} (A)$ of Com-PreLie algebra automorphisms is inducible if and only if there is a
linear map $\varphi:A\longrightarrow V$ satisfying the following conditions:
\begin{eqnarray}
  \label{IamA1-1}  \beta (\mu(x)u ) = \mu(\alpha (x)) \beta (u),
\beta (l(x) u) = l(\alpha (x)) \beta (u),
\beta (r(x) u) = r(\alpha (x)) \beta (u), \\
   \label{IamA1-2}\beta\big(\phi(x,y)\big)-\phi(\alpha(x),\alpha(y)) = \mu(\alpha(x))\varphi(y) -\varphi(x\ast y)+\mu(\alpha(y))\varphi(x) , \\
  \label{IamA1-3} \beta\big(\psi(x,y)\big)-\psi(\alpha(x),\alpha(y)) = l(\alpha(x))\varphi(y) -\varphi(x\bullet y)+r(\alpha(y))\varphi(x) ,
\end{eqnarray}
for all $x\in A, u\in V$.
\end{pro}

\begin{proof}
    Since $(\beta, \alpha)$ is inducible, there exists a Com-PreLie algebra automorphism $\gamma \in \mathrm{Aut}_V (\widehat{A})$ such that $\gamma |_V = \beta$ and $\overline{\gamma} = \alpha$. Since $s$ is a section of $j$,
for all
$x\in A$, we have
$$j(\gamma s-s\alpha)(x)=\alpha(x)-\alpha(a)=0$$
which implies that $(\gamma s-s\alpha)(x)\in \mathrm{ker}(j)=V.$ So we can define linear maps $\varphi:A\longrightarrow
V$ by
\begin{equation} \varphi(x)=(\gamma s-s\alpha)(x),~~\forall~x\in A.\end{equation}
Then for any $a \in A, u \in V$, by the fact that $V$ has trivial Com-PreLie structure, we get
    \begin{align*}
         &\beta (\mu(x)u ) = \gamma (s(x) \ast_{\widehat{A}} u) = \gamma s (x) \ast_{\widehat{A}} \gamma (u)         = s \alpha (x) \ast_{\widehat{A}} \beta (u)  =  \mu(\alpha (x)) \beta (u), \\
        & \beta (l(x) u) = \gamma ( s(x) \bullet_{\widehat{A}} u ) = \gamma s (x) \bullet_{\widehat{A}} \gamma (u)   = s \alpha (x)\bullet_{\widehat{A}} \beta (u)  = l(\alpha (x)) \beta (u),\\
         & \beta (r(x) u) = \gamma (u \bullet_{\widehat{A}}  s(x) ) =\gamma (u) \bullet_{\widehat{A}}\gamma s (x)     = \beta (u) \bullet_{\widehat{A}} s \alpha (x) = r(\alpha (x)) \beta (u),
    \end{align*}
    and
    \begin{eqnarray*}
&&\beta\big(\phi(x,y)\big)-\phi(\alpha(x),\alpha(y)) \\
  &=& \gamma\big(s(x) \ast_{\widehat{A}} s(y) - s (x \ast y)\big)-\big(s(\alpha(x)) \ast_{\widehat{A}} s(\alpha(y)) - s (\alpha(x) \ast \alpha(y))\big)\\
  &=&\gamma s(x) \ast_{\widehat{A}} \gamma s(y)-\gamma  s (x \ast y)+s (\alpha(x) \ast \alpha(y)-s(\alpha(x)) \ast_{\widehat{A}} s(\alpha(y))\\
  &=& \varphi(x)\ast_{\widehat{A}} \varphi(y)+s(\alpha(x)) \ast_{\widehat{A}} \varphi(y)-\varphi(x\ast y)+s(\alpha(y)) \ast_{\widehat{A}} \varphi(x)\\
  &=&\mu(\alpha(x))\varphi(y) -\varphi(x\ast y)+\mu(\alpha(y))\varphi(x),
\end{eqnarray*}
which implies that Eq.~(\ref{IamA1-2}) holds.
Analogously, we can show that Eq.~(\ref{IamA1-3}) holds.

Conversely, suppose that $(\beta, \alpha) \in \mathrm{Aut} (V) \times \mathrm{Aut} (A)$ and there are a linear map $\varphi:A\longrightarrow V$ satisfying Eqs. (\ref{IamA1-1})-(\ref{IamA1-3}). Since $s$ is a section of $j$,
all $\widehat{x}\in \widehat{A}$ can be written as
$\widehat{x}= s(x)+u$ for some $x\in A,u\in V.$
Define linear maps $\gamma:\widehat{A}\longrightarrow
\widehat{A}$ by
\begin{equation} \label{Aut1} \gamma(\widehat{x})=\gamma( s(x)+u)= s(\alpha(x))+\beta(u)+\varphi(x).\end{equation}
The map is injective as $\gamma (\widehat{x}) = 0$ (for $\widehat{x} = s(x) + u$) implies that $s (\alpha (x)) = 0$ and $\beta (u) + \varphi (x) = 0$. Since $s, \alpha, \beta$ are all injective maps, we get that $x= 0$ and $u = 0$ which in turn implies that $\widehat{x} = 0$. The map $\gamma$ is also surjective as $\widehat{x} = s(x) + u \in \widehat{A}$ has unique preimage $\widehat{x}' = s (\alpha^{-1} (x)) + ( \beta^{-1} (u) - \beta^{-1} \varphi \alpha^{-1} (x)) \in \widehat{A}$. This shows that $\gamma$ is indeed a bijective map.

 Next, we claim that $\gamma: \widehat{A} \rightarrow \widehat{A}$ is a Com-PreLie algebra automorphism. For any two elements $ \widehat{x}= s(x) + u$ and $\widehat{y} = s(y) + v$ from the space $\widehat{A}$, we observe that
    \begin{eqnarray*}
        &&\gamma (\widehat{x}) \ast_{\widehat{A}} \gamma (\widehat{y}) \\
        &=& \big(  s (\alpha (x)) + \beta (u) + \varphi (x) \big) \ast_{\widehat{A}} \big(  s (\alpha (y)) + \beta (v) + \varphi (y)   \big) \\
       & =& s (\alpha (x)) \ast_{\widehat{A}} s (\alpha (y)) + s (\alpha (x)) \ast_{\widehat{A}} \beta (v) + s (\alpha (x)) \ast_{\widehat{A}} \varphi(y) + s (\alpha (y)) \ast_{\widehat{A}} \beta (u) + s (\alpha (y)) \ast_{\widehat{A}} \varphi (x) \\
       & =& s (\alpha (x) \ast \alpha (y)) + \phi (\alpha (x), \alpha (y)) + \mu({\alpha (x)}) \beta (v) + \mu({\alpha (x)}) \varphi (y) + \mu({\alpha (y)} ) \beta (u) + \mu({\alpha (y)} )\varphi (x) \\
        &=& s (\alpha (x) \ast \alpha (y)) + \beta \big(   \phi (x, y) + \mu(x) v + \mu(y) u \big) + \varphi (x \ast y) \qquad (\text{by } (\ref{IamA1-1}) \text{ and } (\ref{IamA1-2})) \\
       & =&\gamma \big(  s (x \ast y) + \phi (x, y) + s(x) \ast_{\widehat{A}} v + s(y) \ast_{\widehat{A}} u  \big) \qquad (\text{from the definition of } \gamma)\\
       & =& \gamma \big(  s(x) \ast_{\widehat{A}} s (y) + s (x) \ast_{\widehat{A}} v + s(y) \ast_{\widehat{A}} u  \big) \\
       & =& \gamma (( s(x) + u ) \ast_{\widehat{A}} (s(y) + v) ) = \gamma (\widehat{x} \ast_{\widehat{A}} \widehat{y}).
    \end{eqnarray*}
 In the same way, one can show that $ \gamma (\widehat{x}) \bullet_{\widehat{A}} \gamma (\widehat{y}) = \gamma (\widehat{x}\bullet_{\widehat{A}} \widehat{y} ).$ Hence the claim follows. Finally, it follows from the definition of $\gamma$ that $\gamma (u) = \beta (u)$ and $\overline{\gamma} (x) = j \gamma s (x) = j ( s (\alpha (x)) + \varphi (x) )= \alpha (x)$, for all $x \in A$ and $u \in V$. Hence $\gamma \in \mathrm{Aut}_V (A)$ and $\tau (\gamma) = (\gamma|_V, \overline{\gamma}) = (\beta, \alpha)$ which shows that the pair $(\beta, \alpha)$ is inducible. This completes the proof.
\end{proof}

Assume that (\ref{ab-ext-1})
be an abelian extension of $A$ by $V $
   with a section $s$ of $j$ and
$(\phi, \psi)$ be the corresponding abelian 2-cocycle induced by
$s$. Define
\begin{align*}
		\mathcal{C}=&\left\{\begin{aligned}&(\beta, \alpha)\in \mathrm{Aut}(V
)\times \mathrm{Aut}(A)
\end{aligned}\left|
\begin{aligned}& Eqs.~(\ref{IamA1-1})~ hold
     \end{aligned}\right.\right\}.
 	\end{align*}
 to be the set of all compatible pairs of automorphisms. Then $\mathcal{C}$ is obviously a subgroup of $ \mathrm{Aut}(V
)\times \mathrm{Aut}(A)$. For all $(\beta, \alpha)\in \mathrm{Aut}(V
)\times \mathrm{Aut}(A)$,
define a symmetric bilinear map $ \phi \in \mathrm {Hom} (A\otimes A,V) $  and a bilinear map $ \psi \in \mathrm {Hom} (A\otimes A,V) $  respectively by
 \begin{align}
 \phi_{(\beta, \alpha)}(x,y)=\beta \phi(\alpha^{-1}(x),\alpha^{-1}(y)),\\
 \psi_{(\beta, \alpha)}(x,y)=\beta \psi(\alpha^{-1}(x),\alpha^{-1}(y)),
\end{align}
for all $ x,y\in A$.

We denote $\big( \phi_{(\beta, \alpha)}, \psi_{(\beta, \alpha)} \big)$
by $(\phi,\psi)^{(\beta, \alpha)}$ for simplicity.
In general, $(\phi,\psi)^{(\beta, \alpha)}$ may not be a 2-cocycle.
In fact, $(\phi,\psi)^{(\beta, \alpha)}$ is a 2-cocycle if $(\beta, \alpha)\in \mathcal{C}$.

\begin{pro} With the above notations, if $(\beta, \alpha)\in \mathcal{C}$, then $(\phi,\psi)^{(\beta, \alpha)}$ is an abelian 2-cocycle.
\end{pro}
\begin{proof}
  Since $(\beta, \alpha)\in \mathcal{C}$, for any $x,y\in A$, we have the Eqs.~ (\ref{IamA1-1}) hold.
Given that $(\phi,\psi)$ is a 2-cocycle, the identities (\ref{A2co-1})-(\ref{A2co-3}) are hold.
In identities (\ref{A2co-3}), if we replace $x,y,z$ by $\alpha^{-1}(x),\alpha^{-1}(y), \alpha^{-1}(z)$ respectively, we obtain
 \begin{eqnarray*}
 \beta\big(\psi(\alpha^{-1}(x), \alpha^{-1}(y) \ast \alpha^{-1}(z))+l(\alpha^{-1}(x))\phi(\alpha^{-1}(y),\alpha^{-1}(z))\\
-\phi(\alpha^{-1}(x) \bullet \alpha^{-1}(y),\alpha^{-1}(z))-\mu(\alpha^{-1}(z))\psi(\alpha^{-1}(x),\alpha^{-1}(y))\notag\\
-\phi(\alpha^{-1}(y),\alpha^{-1}(x) \bullet \alpha^{-1}(z))-\mu(\alpha^{-1}(y))\psi(\alpha^{-1}(x),\alpha^{-1}(z))\big)=0
 \end{eqnarray*}
   which can be written as
  \begin{eqnarray*}
\psi_{(\beta, \alpha)}(x, y \ast z)+l(x)\phi_{(\beta, \alpha)}(y,z)-\phi_{(\beta, \alpha)}(x \bullet y,z)-\mu(z) \psi_{(\beta, \alpha)}(x,y)\notag\\
-\phi_{(\beta, \alpha)}(y,x \bullet z)-\mu(y) \psi_{(\beta, \alpha)}(x,z)=0.
  \end{eqnarray*}
  which implies that Eq.~(\ref{A2co-3}) holds for $(\phi,\psi)^{(\beta, \alpha)}$. Similarly, we can check that
Eq.~(\ref{A2co-1}) and Eq.~(\ref{A2co-2}) hold. This finishes the proof.
\end{proof}

Therefore, an equivalent formulation of the Theorem \ref{Idu} is presented below.

\begin{thm}\label{CIdu}
 A pair $ (\beta, \alpha)$ is inducible if and only if $(\beta, \alpha)\in \mathcal{C}$ and the two
 abelian 2-cocycles  $(\phi,\psi)^{(\beta, \alpha)}$ and  $(\phi,\psi)$
 are equivalent.
\end{thm}
\begin{proof}
Let $(\beta, \alpha)$ is inducible. It is obviously that $(\beta, \alpha)\in \mathcal{C}$. Then by Theorem \ref{Idu}, there exists linear maps $\varphi:A\longrightarrow V$
satisfying Eqs.~ (\ref{IamA1-2})-(\ref{IamA1-3}).
 We replace $ x, y$ by $\alpha^{-1}(x),\alpha^{-1}(y)$ respectively, we get
 \begin{eqnarray*}
\phi_{(\beta, \alpha)}(x,y)-\phi(x,y) = \mu(x)\varphi\alpha^{-1}(y) -\varphi\alpha^{-1}(x\ast y)+\mu(y)\varphi(\alpha^{-1}(x)) , \\
\psi_{(\beta, \alpha)}(x,y)-\psi(x,y) = l(x)\varphi\alpha^{-1}(y) -\varphi\alpha^{-1}(x\bullet y)+r(y)\varphi(\alpha^{-1}(x)) .
 \end{eqnarray*}
 This shows that  $(\phi,\psi)^{(\beta, \alpha)}$ and  $(\phi,\psi)$ are equivalent and the equivalence is given by the map $\varphi \alpha^{-1}  \in \mathrm{Hom}(A,V).$

  Conversely, suppose that $(\phi,\psi)^{(\beta, \alpha)}$ and  $(\phi,\psi)$ are equivalent and the equivalence is given by the map $f \in \mathrm{Hom}(A,V),$ i.e.
  \begin{eqnarray}
   \label{cobound-11}\beta \phi(\alpha^{-1}(x),\alpha^{-1}(y)) - \phi (x, y) = \mu(x) f (y)- f (x \ast y ) + \mu(y) f (x), \\
   \label{cobound-21}\beta \psi(\alpha^{-1}(x),\alpha^{-1}(y)) - \psi (x, y) =l(x)f (y) - f (x\bullet y ) + r(y) f (x),
\end{eqnarray}
for all $x,y\in A$. We now define a map $\varphi:=f \alpha:A\longrightarrow V$,
 then by replace $ x, y$ by $\alpha(x),\alpha(y)$ in Eq.~(\ref{cobound-11}) and Eq.~(\ref{cobound-21})  respectively, it can be easily checked that the map $\varphi$  satisfies the Eqs.~ (\ref{IamA1-2})-(\ref{IamA1-3}). Thus the pair $(\beta, \alpha)$ is inducible.
 \end{proof}
\section{Wells exact sequences}\label{sec-5}
In this section, we consider the Wells map associated with abelian extensions of a Com-PreLie algebra.

 Let (\ref{ab-ext-1})
be an abelian extension of $A$ by $V$
   with a section $s$ of $j$ and
$(\phi,\psi)$ be the corresponding abelian 2-cocycle induced by $s$.

Define a  map $\mathcal {W}:\mathcal{C}\longrightarrow \mathcal{H}^2 (A ,V)$ by
\begin{equation}\label{W1}
	\mathcal {W}(\beta,\alpha)
=[(\phi,\psi)^{(\beta,\alpha)}
-(\phi,\psi)].
\end{equation}
The map $\mathcal {W}$ is called the Wells map associated with $\mathcal{E}$.

\begin{pro}
 The Wells map  $\mathcal {W}$  does not depend on the choice of section.
\end{pro}
\begin{proof}

 Let $s$ and $s'$ be two
sections of $j$. The section  $s$  induces the abelian 2-cocycle $(\phi,\psi)$. Likewise, the abelian 2-cocycle  $(\phi',\psi')$
be the corresponding  abelian 2-cocycle induced by the section $s'$. Define linear maps $\varphi:
A\longrightarrow V$ by $\varphi(x)=s(x)-s'(x)$. Since
$j \varphi(x)=j s(x)-j s'(x)=0$,
 $\varphi$ is well defined. Then we have the induced abelian 2-cocycles $(\phi,\psi)$ and $(\phi',\psi')$ are equivalent, and equivalence is given by the map $\varphi$. Since for any $x, y\in A$, we have
 \begin{eqnarray*}
  &&\phi (x, y) - \phi' (x, y) \\
  &=&  s(x) \ast_{\widehat{A}} s(y) - s (x \ast y) - s'(x) \ast_{\widehat{A}} s'(y) + s' (x \ast y) \\
  &=&\mu(x) \varphi (y) - \varphi (x \ast y )+ \mu(y) \varphi (x)- \varphi(x)\ast_{\widehat{A}} \varphi(y)\\
  &=&\mu(x) \varphi (y) - \varphi (x \ast y )+ \mu(y) \varphi (x),
  \end{eqnarray*}
and similarly, $ \psi (x, y) - \psi' (x, y) =l(x) \varphi (y) - \varphi (x\bullet y ) + r(y) \varphi(x).$

On the other hand, the abelian 2-cocycles $(\phi,\psi)^{(\beta,\alpha)}$ and $(\phi',\psi')^{(\beta,\alpha)}$ are equivalent, and equivalence is given by $\beta \varphi \alpha^{-1}$. Since $(\phi,\psi)$ and $(\phi',\psi')$ are equivalent, for any $x, y\in A$, we have
 \begin{eqnarray*}
  &&\phi_{(\beta,\alpha)} (x, y) - \phi'_{(\beta,\alpha)} (x, y) \\
  &=& \beta\big(\mu(\alpha^{-1}(x)\varphi(\alpha^{-1}(y))-\varphi(\alpha^{-1}(x)\ast \alpha^{-1}(y))+\mu(\alpha^{-1}(y))\varphi(\alpha^{-1}(x))\big) \\
  &=&\mu(x)\beta\varphi \alpha^{-1}(y)-\beta\varphi \alpha^{-1}(x\ast y)+\mu(y)\beta\varphi \alpha^{-1}(x),(\because ~ (\beta,\alpha)\in\mathcal{C})
  \end{eqnarray*}
and similarly,
$ \psi_{(\beta, \alpha)}(x,y)-\psi'_{(\beta, \alpha)}(x,y)=l(x) \beta\varphi \alpha^{-1} (y) - \beta\varphi \alpha^{-1} (x\bullet y ) + r(y) \beta\varphi \alpha^{-1}(x).$

Combining the results of the last two paragraphs, we have that the abelian 2-cocycles $$(\phi,\psi)^{(\beta,\alpha)}- (\phi,\psi)$$
and $$(\phi',\psi')^{(\beta,\alpha)}-(\phi',\psi')$$
are equivalent, and equivalence is given by $(\beta\varphi \alpha^{-1}- \varphi)$. Hence they corresponds to the same element in $\mathcal{H}^2_\mathrm{MPL} (A , V )$. This completes the proof.
\end{proof}
The above claim provides an alternative phrasing of the Theorem \ref{Idu} and Theorem \ref{CIdu}: a pair
 $(\beta,\alpha)\in \mathrm{Aut}(A
)\times \mathrm{Aut}(V )$ is inducible if and only if $(\beta,\alpha)\in \mathcal{C}$ and $\mathcal {W}(\beta,\alpha)=0$.

Let
\begin{align*}
\mathrm{Aut}_{V}^{A}
(\widehat{A})=\{\gamma \in \mathrm{Aut}(\widehat{A})| \tau(\gamma)=(\mathrm{id} _{A},\mathrm{id}_{V})\}.
 \end{align*} Recall that
\begin{align*}
		 \mathcal{Z}^1 (A,V)=&\left\{\varphi:A\rightarrow V \left|\begin{aligned}& D^1(\varphi)=0
     \end{aligned}\right.\right\}.
     \end{align*}
It is easy to check that $ \mathcal{Z}^1(A,V)$ is an abelian group, which is called an abelian 1-cocycle on $A $ with values in $V $.
 \begin{pro}\label{Der}
Let (\ref{ab-ext-1})
be an abelian extension of $A$ by $V $. Then $ \mathcal{Z}^1 (A ,V) \cong \mathrm{Aut}_{V}^{A}(\widehat{A})$ as groups.
\end{pro}
\begin{proof}
Define $\Psi:\mathrm{Aut}_{V}^{A}(\widehat{A}) \longrightarrow \mathcal{Z}^1(A ,V)$  by
  $\Psi(\gamma)$, where $\Psi(\gamma) \triangleq \varphi_{\gamma}:A \to V$  is given by
  $$\Psi(\gamma)(x)=\varphi_{\gamma}(x)=\gamma s(x)-s(x), $$
  for all $\gamma\in \mathrm{Aut}_{V}^{A }(\widehat{A}),~x \in A.$
Firstly, we prove that  $\Psi$ is well-defined.
\begin{eqnarray*}
 &&l(x)\varphi_{\gamma}(y)-\varphi_{\gamma}(x \bullet y)+r(y)\varphi_{\gamma}(x)\\
&=&s(x) \bullet \varphi_{\gamma}(y)-\varphi_{\gamma}(x \bullet y)+\varphi_{\gamma}(x) \bullet s(y)\\
&=&-\gamma s(x)\bullet_{\widehat{A}} \gamma s(y)+\gamma s(x) \bullet_{\widehat{A}} s(y)+ s(x) \bullet_{\widehat{A}} \gamma s(y)-s(x) \bullet_{\widehat{A}} s(y)\\
&&+\gamma\big( s(x) \bullet_{\widehat{A}} s(y)\big) - \gamma s (x \bullet y)-s(x) \bullet_{\widehat{A}} s(y) + s (x \bullet y)\\
&=&-\varphi_{\gamma}(x) \bullet \varphi_{\gamma}(y)+\gamma \psi(x,y)-\psi(x,y)(\because ~ \gamma \big|_{V} =\mathrm{id}_{V} ) \\
&=&0,
\end{eqnarray*}

This confirms that Eq.~(\ref{A1co-2}) is applicable.
By the same token, we can prove that Eq.~(\ref{A1co-1}) holds.
Thus, $\varPsi$ is well-defined.

Secondly, for any $\gamma,\gamma'\in \mathrm{Aut}_{V}^{A }(\widehat{A})$ and $x\in A$, suppose $\Psi(\gamma)=\varphi_{\gamma}$
and $\Psi(\gamma')=\varphi_{\gamma'}.$
We get
\begin{align*}\Psi(\gamma \gamma')(x)&=\gamma \gamma's(x)-s(x)
\\&=\gamma(\varphi_{\gamma'}(x)+s(x))-s(x)
\\&=\gamma \varphi_{\gamma'}(x)+\gamma s(x)-s(x)
\\&=\varphi_{\gamma}(x)+\varphi_{\gamma'}(x) \quad (\because ~ \gamma |_V = \mathrm{id}_V)
\\&=\Psi(\gamma)(x)+\Psi(\gamma')(x).\end{align*}
Thus, $\Psi(\gamma \gamma')=\Psi(\gamma)+\Psi(\gamma')$, that is, $\Psi$ is a homomorphism of groups.

Finally, we prove that $\Psi$ is bijective.  For all $\gamma \in \mathrm{Aut}_{V}^{A}(\widehat{A})$, if $\Psi(\gamma)=\varphi_{\gamma}=0$, we can get $\varphi_{\gamma}(x)=\gamma s(x)-s(x)=0$, Therefore, for any $\widehat{x}=s(x) +u \in \widehat{A}$, we have
$$ \gamma(\widehat{x}) = \gamma( s(x)+u)=s(x)+\gamma s(x)-s(x)+\gamma(u)=s(x)+u=\widehat{x},$$
that is, $\gamma=\mathrm{id}_{\hat{A}}$.
Thus, $\Psi$ is injective. To show that the map $\Psi$ is surjective,
 for any $\varphi \in \mathcal{Z}^1(A ,V)$, define linear maps $\gamma:\widehat{A} \rightarrow \widehat{A} $ by
  \begin{equation}\label{W7}\gamma(\widehat{x})=\gamma(s(x)+u)=s(x)+\varphi(x)+u,~\forall~\widehat{x}\in \widehat{A}.\end{equation}
We need to verify that $\gamma $ is an automorphism of the Com-PreLie algebra $\widehat{A}$.
It is obviously that $\gamma$ are bijective. Let $\widehat{x}, \widehat{y} \in \widehat{A}$ , it is not hard to see that $ \gamma( \widehat{x} \ast_{\widehat{A}} \widehat{y}) =\gamma  (\widehat{x}) \ast_{\widehat{A}} \gamma (\widehat{y}) $,$ \gamma(\widehat{x} \bullet_{\widehat{A}} \widehat{y}) =\gamma(\widehat{x})\bullet_{\widehat{A}} \gamma(\widehat{y})$.
Thus, $\gamma$ is an isomorphism of the Com-PreLie algebra, i.e., $\gamma\in  \mathrm{Aut}(\widehat{A})$.
 Moreover, we have $(j\gamma s,\gamma|_{V})=(\mathrm{id}_A,\mathrm{id}_V)$. It follows that $\gamma \in \mathrm{Aut}_{V}^{A}(\widehat{A})$. Thus, $\Psi$ is surjective. In all, $\Psi$ is bijective.
 So, $\mathcal{Z}^1(A ,V) \cong \mathrm{Aut}_{V}^{A}(\widehat{A})$.
\end{proof}
\begin{thm} Let
 (\ref{ab-ext-1})
be an abelian extension of $A$ by $V$. There is an exact sequence:
$$0\longrightarrow \mathcal{Z}^1(A,V)\stackrel{\iota}{\longrightarrow} \mathrm{Aut}_{V}(\widehat{A})
\stackrel{\tau}{\longrightarrow}\mathcal{C} \stackrel{\mathcal {W}}{\longrightarrow} \mathcal{H}^2(A,V).$$
\end{thm}
\begin{proof}Since the inclusion map $\iota : \mathcal{Z}^1 (A ,V) \to \mathrm{Aut}_{V}(\widehat{A})$ is an injection, the above sequence is exact at the first term.

Next, since $\mathrm{ker} (\tau)=\mathrm{Aut}_{V}^{A}(\widehat{A})$.
By Theorem \ref{Der}, we have $\mathrm{Aut}_{V}^{A}(\widehat{A})\cong \mathcal{Z}^1(A ,V).$
Thus, we have $\mathrm{ker} (\tau) =\mathrm{im}(\iota)$. This shows that the sequence is exact at the second term.

Finally, to show that the sequence is exact at the third term, take a pair $(\beta,\alpha) \in \mathrm{ker}(\mathcal{W})$, that is $\mathcal {W}(\beta,\alpha)=0$. Thus, we have the abelian 2-cocycles $(\phi,\psi)^{(\beta,\alpha)}$ and $(\phi,\psi)$ are equivalent, by Theorem \ref{CIdu}, the pair $(\beta,\alpha)$ is inducible. In other words, there exists an automorphism $\gamma \in \mathrm{Aut}_{V}(\widehat{A})$ such that $\tau(\gamma) = (\beta,\alpha)$. This shows that $(\beta,\alpha) \in \mathrm{im}(\tau)$. Conversely, if a pair $(\beta,\alpha) \in \mathrm{im}(\tau)$, then by definition the pair $(\beta,\alpha)$ is inducible. Hence again by Theorem \ref{CIdu}, the abelian 2-cocycles $(\phi,\psi)^{(\beta,\alpha)}$ and $(\phi,\psi)$ are equivalent. Therefore, $\mathcal{W}(\beta,\alpha) = 0$ which implies that $(\beta,\alpha) \in \mathrm{ker}(\mathcal{W})$. Thus, we obtain $\mathrm{ker}(\mathcal{W}) = \mathrm{im}(\tau)$. Hence the result follows.
\end{proof}

\vskip7pt
\footnotesize{
\noindent Tao Zhang\\
School of Mathematics and Statistics,\\
Henan Normal University, Xinxiang 453007, P. R. China;\\
 E-mail address: \texttt{{zhangtao@htu.edu.cn}}

\vskip7pt
\footnotesize{
\noindent Ying-Hua Lu\\
School of Mathematics and Statistics,\\
Henan Normal University, Xinxiang 453007, P. R. China;\\
 E-mail address: \texttt{{lu-yinghua@qq.com}}

 \end{document}